\documentclass{amsart}
\usepackage{amsmath, amssymb, amsthm, epsfig}
\usepackage{amssymb, latexsym}
\usepackage{url}

\DeclareMathOperator{\GL}{\mathrm GL}

\DeclareMathOperator{\Norm}{\mathrm Norm}

\DeclareMathOperator{\rank}{\mathrm rank}

\DeclareMathOperator{\sgn}{\mathrm sgn}
\DeclareMathOperator{\spann}{\mathrm span}

\DeclareMathOperator{\Tor}{\mathrm Tor}

\DeclareMathOperator{\Vol}{\mathrm Vol}

\newtheorem{thm}{Theorem}[section]
\newtheorem{cor}[thm]{Corollary}
\newtheorem{lem}[thm]{Lemma}
\newtheorem{prop}[thm]{Proposition}

\newtheorem{ex}{Example}

\begin{document}
\bibliographystyle{plain}

 \newtheorem{theorem}{Theorem}[section]
 \newtheorem{lemma}{Lemma}[section]
 \newtheorem{corollary}{Corollary}[section]
 \newtheorem{conjecture}{Conjecture}[section]
 \newcommand{\mc}{\mathcal}
 \newcommand{\A}{\mc A}
 \newcommand{\B}{\mc B}
 \newcommand{\cc}{\mc C}
 \newcommand{\D}{\mc D}
 \newcommand{\E}{\mc E}
 \newcommand{\F}{\mc F}
 \newcommand{\G}{\mc G}
 \newcommand{\hH}{\mc H}
 \newcommand{\I}{\mc I}
 \newcommand{\J}{\mc J}
 \newcommand{\eL}{\mc L}
 \newcommand{\M}{\mc M}
 \newcommand{\eN}{\mc N}
 \newcommand{\pP}{\mc P}
 \newcommand{\qq}{\mc Q}
 \newcommand{\sS}{\mc S}
 \newcommand{\U}{\mc U}
 \newcommand{\V}{\mc V}
 \newcommand{\X}{\mc X}
 \newcommand{\Y}{\mc Y}
 \newcommand{\Z}{\mc Z}
 \newcommand{\aA}{\mathbb A}
 \newcommand{\C}{\mathbb C}
 \newcommand{\R}{\mathbb R}
 \newcommand{\N}{\mathbb N}
 \newcommand{\bP}{\mathbb P}
 \newcommand{\Q}{\mathbb Q}
 \newcommand{\T}{\mathbb T}
 \newcommand{\zZ}{\mathbb Z}
 \newcommand{\fA}{\mathfrak A}
 \newcommand{\fB}{\mathfrak B}
 \newcommand{\ff}{\mathfrak F}
 \newcommand{\fI}{\mathfrak I}
 \newcommand{\fJ}{\mathfrak J}
 \newcommand{\fP}{\mathfrak P}
 \newcommand{\fQ}{\mathfrak Q}
 \newcommand{\fU}{\mathfrak U}
 \newcommand{\fb}{f_{\beta}}
 \newcommand{\fg}{f_{\gamma}}
 \newcommand{\gb}{g_{\beta}}
 \newcommand{\vphi}{\gamma}
 \newcommand{\vep}{\varepsilon}
 \newcommand{\bo}{\boldsymbol 0}
 \newcommand{\bone}{\boldsymbol 1}
 \newcommand{\ba}{\boldsymbol a}
 \newcommand{\bb}{\boldsymbol b}
 \newcommand{\bc}{\boldsymbol c}
 \newcommand{\be}{\boldsymbol e}
 \newcommand{\bk}{\boldsymbol k}
 \newcommand{\bm}{\boldsymbol m}
 \newcommand{\bn}{\boldsymbol n}
 \newcommand{\balpha}{\boldsymbol \alpha}
 \newcommand{\Bbeta}{\boldsymbol \beta}
 \newcommand{\bgamma}{\boldsymbol \gamma}
 \newcommand{\bbeta}{\boldsymbol \eta}
 \newcommand{\blambda}{\boldsymbol \lambda}
 \newcommand{\bpsi}{\boldsymbol \psi}
 \newcommand{\bomega}{\boldsymbol \omega}
 \newcommand{\bxi}{\boldsymbol \xi}
 \newcommand{\bzeta}{\boldsymbol \zeta}
 \newcommand{\bt}{\boldsymbol t}
 \newcommand{\bu}{\boldsymbol u}
 \newcommand{\bv}{\boldsymbol v}
 \newcommand{\bx}{\boldsymbol x}
 \newcommand{\bwy}{\boldsymbol y}
 \newcommand{\bw}{\boldsymbol w}
 \newcommand{\bz}{\boldsymbol z}
 \newcommand{\hmu}{\widehat \mu}
 \newcommand{\oK}{\overline{K}}
 \newcommand{\oKt}{\overline{K}^{\times}}
 \newcommand{\oQ}{\overline{\Q}}
 \newcommand{\oq}{\oQ^{\times}}
 \newcommand{\oQt}{\oQ^{\times}/\Tor\bigl(\oQ^{\times}\bigr)}
 \newcommand{\ot}{\Tor\bigl(\oQ^{\times}\bigr)}
 \newcommand{\h}{\frac12}
 \newcommand{\hh}{\tfrac12}
 \newcommand{\dx}{\text{\rm d}x}
 \newcommand{\dbx}{\text{\rm d}\bx}
 \newcommand{\dy}{\text{\rm d}y}
 \newcommand{\dmu}{\text{\rm d}\mu}
 \newcommand{\dnu}{\text{\rm d}\nu}
 \newcommand{\dla}{\text{\rm d}\lambda}
 \newcommand{\dlav}{\text{\rm d}\lambda_v}
 \newcommand{\trho}{\widetilde{\rho}}
 \newcommand{\dtrho}{\text{\rm d}\widetilde{\rho}}
 \newcommand{\drho}{\text{\rm d}\rho}
 

\def\A{{\mathcal A}}
\def\B{{\mathcal B}}
\def\C{{\mathcal C}}
\def\D{{\mathcal D}}
\def\F{{\mathcal F}}
\def\x{{\mathcal H}}
\def\I{{\mathcal I}}
\def\J{{\mathcal J}}
\def\K{{\mathcal K}}
\def\L{{\mathcal L}}
\def\M{{\mathcal M}}
\def\N{{\mathcal N}}
\def\O{{\mathcal O}}
\def\R{{\mathcal R}}
\def\s{{\mathcal S}}
\def\V{{\mathcal V}}
\def\W{{\mathcal W}}
\def\X{{\mathcal X}}
\def\Y{{\mathcal Y}}
\def\H{{\mathcal H}}
\def\Z{{\mathcal Z}}
\def\OO{{\mathcal O}}
\def\BB{{\mathbb B}}
\def\cee{{\mathbb C}}
\def\pee{{\mathbb P}}
\def\que{{\mathbb Q}}
\def\real{{\mathbb R}}
\def\zed{{\mathbb Z}}
\def\hyp{{\mathbb H}}
\def\aa{{\mathfrak a}}
\def\HH{{\mathfrak H}}
\def\qbar{{\overline{\mathbb Q}}}
\def\eps{{\varepsilon}}
\def\ahat{{\hat \alpha}}
\def\bhat{{\hat \beta}}
\def\gt{{\tilde \gamma}}
\def\h{{\tfrac12}}
\def\be{{\boldsymbol e}}
\def\bei{{\boldsymbol e_i}}
\def\bff{{\boldsymbol f}}
\def\ba{{\boldsymbol a}}
\def\bb{{\boldsymbol b}}
\def\bc{{\boldsymbol c}}
\def\bm{{\boldsymbol m}}
\def\bk{{\boldsymbol k}}
\def\bi{{\boldsymbol i}}
\def\bl{{\boldsymbol l}}
\def\bq{{\boldsymbol q}}
\def\bu{{\boldsymbol u}}
\def\bt{{\boldsymbol t}}
\def\bs{{\boldsymbol s}}
\def\bv{{\boldsymbol v}}
\def\bw{{\boldsymbol w}}
\def\bx{{\boldsymbol x}}
\def\bX{{\boldsymbol X}}
\def\bz{{\boldsymbol z}}
\def\bwy{{\boldsymbol y}}
\def\bY{{\boldsymbol Y}}
\def\bL{{\boldsymbol L}}
\def\baa{{\boldsymbol\alpha}}
\def\bbb{{\boldsymbol\beta}}
\def\bet{{\boldsymbol\eta}}
\def\bxi{{\boldsymbol\xi}}
\def\bo{{\boldsymbol 0}}
\def\bol{{\boldkey 1}_L}
\def\ep{\varepsilon}
\def\p{\boldsymbol\varphi}
\def\q{\boldsymbol\psi}
\def\rank{\operatorname{rank}}
\def\aut{\operatorname{Aut}}
\def\lcm{\operatorname{lcm}}
\def\sgn{\operatorname{sgn}}
\def\spn{\operatorname{span}}
\def\md{\operatorname{mod}}
\def\Norm{\operatorname{Norm}}
\def\dim{\operatorname{dim}}
\def\det{\operatorname{det}}
\def\Vol{\operatorname{Vol}}
\def\rk{\operatorname{rk}}
\def\Gal{\operatorname{Gal}}
\def\WR{\operatorname{WR}}
\def\WO{\operatorname{WO}}
\def\GL{\operatorname{GL}}
\def\pr{\operatorname{pr}}


\newcommand{\multicom}[1]{}
  
\title{On a new absolute version of Siegel's lemma}
\author{Maxwell Forst}
\author{Lenny Fukshansky}\thanks{Fukshansky was partially supported by the Simons Foundation grant \#519058}
\subjclass[2020]{11G50, 11H06, 11D99}
\keywords{Siegel's lemma, heights}

\address{Institute of Mathematical Sciences, Claremont Graduate University, Claremont, CA 91711 USA}
\email{maxwell.forst@cgu.edu}
\address{Department of Mathematics, 850 Columbia Avenue, Claremont McKenna College, Claremont, CA 91711 USA}
\email{lenny@cmc.edu}

\numberwithin{equation}{section}

\begin{abstract} 
We establish a new version of Siegel's lemma over a number field $k$, providing a bound on the maximum of heights of basis vectors of a subspace of $k^N$, $N \geq 2$. In addition to the small-height property, the basis vectors we obtain satisfy certain sparsity condition. Further, we produce a nontrivial bound on the heights of all the possible subspaces generated by subcollections of these basis vectors. Our bounds are absolute in the sense that they do not depend on the field of definition. The main novelty of our method is that it uses only linear algebra and does not rely on the geometry of numbers or the Dirichlet box principle employed in the previous works on this subject.
\end{abstract}

\maketitle

\section{Introduction and statement of results}
\label{intro}

  The name Siegel's lemma is usually attributed to results on small-size nontrivial solutions to underdetermined systems of homogeneous linear equations over a field or a ring of arithmetic interest.
  The original Siegel's lemma asserts that given an integer $M \times N$ matrix $A$, $M < N$, there exists a nonzero $\bxi \in \zed^N$ such that $A\bxi = \bo$ and
  \begin{equation}\label{basic_siegel}
    |\bxi| \leq 2 + \left( N |A| \right)^{\frac{1}{N-M}},
  \end{equation}
  where $|\cdot |$ here denotes the sup-norm (maximum of the absolute values of the coordinates) of the vector $\bxi$ and the matrix $A$, respectively. 
  This result was originally used by Thue~\cite{thue1909} and Siegel~\cite{siegel1929} to establish measures of irrationality for certain algebraic numbers, specifically to prove the existence of polynomials with small integer coefficients which vanish with prescribed multiplicity at given algebraic points. Since then, Siegel's lemma has found multiple applications in construction of auxiliary polynomials with integer coefficients that vanish at prescribed algebraic points, see the recent book~\cite{masser2016} by D. W. Masser for details. The exponent $\frac{1}{N-M}$ in~\eqref{basic_siegel} is known to be best possible, however this bound lacks invariance under linear transformations: indeed, for any $M \times M$ integer matrix $U$,
  $$(UA)\bxi = A\bxi = \bo,$$
  however $|UA|$ and $|A|$ can be very different.
  It therefore makes sense to rephrase this fundamental principle as a result on the existence of points of bounded size in a vector space, i.e. the $L$-dimensional null-space of the given $N \times M$ linear system, where $L = N-M$.
  In this context, ``size" is usually measured via a suitable {\it height function}, a standard tool of arithmetic geometry.

  The first subspace version of Siegel's lemma over an arbitrary number field $k$, providing a full small-height basis, was established by Bombieri and Vaaler in~\cite{bombieri1983}. A convenient for our purposes formulation, which easily follows from the results of ~\cite{bombieri1983}, asserts that, given an $L$-dimensional subspace $\Z$ of $k^N$ there exists a basis $\bxi_1,\dots,\bxi_L$ for $\Z$ such that
  \begin{equation}\label{bombieri_vaaler}
    \prod_{i=1}^L H(\bxi_i) \leq N^{L/2} \left( \left( \frac{2}{\pi} \right)^{r_2} |\Delta_k| \right)^{\frac{L}{2d}} H(\Z),
  \end{equation}
  where $d = [k:\que]$, $r_2$ is the number of pairs of complex conjugate embeddings of $k$, $\Delta_k$ is the discriminant of $k$, and $H$ is a height function (we review all the necessary notation including heights, as defined in \cite{bombieri2006} and \cite{vaaler2003}, in Section~\ref{notation}).

  Our goal is to establish a new version of Siegel's lemma with somewhat different bounds on the maximum of heights of the basis vectors, independent of the field of definition. 
  Throughout this note, we assume that $k$ is an algebraic number field and work in the $k$-linear space of $N \times 1$ column vectors $k^N$. 
  More generally, we could also work in an intermediate field $K$ such that $\Q \subseteq K \subseteq \oQ$ where it may happen that $K/\Q$ is an extension of infinite degree. 
  However, we are interested in simultaneous solutions to finitely many linear equations having algebraic numbers as coefficients. 
  Such systems involve only finitely many algebraic numbers and these generate a finite extension of $\Q$. 
  Several inequalities in the literature (for example, \cite[Theorem 8 and Theorem 9]{bombieri1983} and \cite[Theorem 1, and Theorem 2]{vaaler2003}) that bound the height of solutions to simultaneous systems of linear equations with coefficients in a number field $k$ contain constants that depend on $k$.  
  Usually these constants depend on the discriminant of $k$, as in~\eqref{bombieri_vaaler} above.  
  An exception to this situation can be found in the striking results of Roy and Thunder \cite{roy1996}, and \cite{roy1999}, on absolute forms of Siegel's Lemma (a similar result follows from the celebrated paper of S. Zhang~\cite{zhang1995}, however the Roy-Thunder notation is more convenient for our purposes). Analogously to~\eqref{bombieri_vaaler}, they prove the existence of a basis $\bxi_1,\dots,\bxi_L$ for an $L$-dimensional subspace $\Z \subset \qbar^N$ with
  \begin{equation}\label{roy_thunder}
    \prod_{i=1}^L H(\bxi_i) \leq \left( e^{\frac{L(L-1)}{4}} + \eps \right) H(\Z),
  \end{equation}
  for any $\eps > 0$ (the choice of the basis depends on $\eps$). 
  While their bound does not depend on any number field, the vectors $\bxi_1,\dots,\bxi_L$ are also not guaranteed to lie over a fixed number field.

 Bridging between \cite{bombieri1983} and the work of Roy and Thunder, we establish the existence of a small-height basis for an $L$-dimensional subspace of $k^N$ (i.e., the space of solutions to a system of simultaneous linear equations), and the inequalities we prove are free of constants that depend on a number field. 
  While we bound the individual heights of the vectors instead of the product, our basis lies over a fixed number field $k$ and our bound is particularly simple. 
  In fact, we prove more than just the existence of a small-height basis for a subspace. A vector $\bx \in k^N$ is called {\it $s$-sparse} for an integer $1 \leq s \leq N$ if $\bx$ has no more than $s$ nonzero coordinates. Here is our main result, which exhibits a small-height basis with additional sparsity and ``monotonicity" properties.
 
  \begin{theorem}  \label{new_siegel} 
    Let $\Z = Ak^L \subseteq k^N$ be a subspace of dimension $L$ where $1 \leq L < N$ and $A$ is an $N \times L$ basis matrix for $\Z$. Let $B$ be a full rank $L \times L$ submatrix of $A$, and write 
    \begin{equation*}\label{siegel119}
      \big\{\bomega_1, \dots, \bomega_L\big\}
    \end{equation*}
for the column vectors of the matrix $AB^{-1}$. Then $\bomega_1, \dots, \bomega_L$ is another basis for $\Z$ over $k$, which consists of $(N-L+1)$-sparse vectors satisfying the following property: if $I \subseteq \{1, \dots , L\}$ is a nonempty subset, and 
    \begin{equation}\label{siegel121}
      \Y_I = \spann_k \big\{\bomega_i : i \in I\big\}
    \end{equation}
    is the $k$-linear subspace spanned by the basis vectors that are indexed by the elements in $I$, then
    \begin{equation}\label{siegel123}
      H(\Y_I) \le H(\Z).
    \end{equation}
    In particular, $\max_{1 \leq i \leq L} H(\bomega_i) \leq H(\Z)$.
    Moreover, if \(I_1 \subsetneq I_2 \subseteq \{1,\ldots, L\}\), then 
    \begin{equation}\label{siegelm113}
      H(\Y_{I_1}) \leq H(\Y_{I_2}).
    \end{equation}
    Equality is attained in \eqref{siegelm113} if and only if \(\bomega_i\) is a standard basis vector for each \(i \in I_2\setminus I_1\).
  \end{theorem}

We illustrate Theorem~\ref{new_siegel} by Example~\ref{ex_main} in Section~\ref{newSiegelProof}. Observe that, while our result does not imply the bounds~\eqref{bombieri_vaaler} and~\eqref{roy_thunder}, those bounds do not imply our result either. Further, notice that in situations when height of the subspace $\Z$ is dominated by the constant depending on $N$ and $k$ in~\eqref{bombieri_vaaler} or on $L$ in~\eqref{roy_thunder} our bound may be better. In fact, the constant in~\eqref{bombieri_vaaler} has been improved in~\cite{vaaler2003} (we review this result in Section~\ref{notation}), but even this optimal constant depends on a power of $\Delta_k$. Additionally, our bound can be preferable in some applications due to its simplicity. For example, in some situations it is important to obtain a point in a vector space $\Z$ outside of a co-dimension one subspace $\Z'$ (see~\cite{lenny_JNT2010} and references therein for a detailed discussion of Siegel's lemma-type results with avoidance conditions). It is clear that at least one of the basis vectors for $\Z$ is not in $\Z'$, however we do not know which vector it is. If we are to use the Bombieri-Vaaler basis satisfying~\eqref{bombieri_vaaler}, we can only assume that this vector's height is bounded by the right-hand side of that equation: the vector in question may be the largest with respect to height out of the basis, and the rest of the basis vectors may have height equal to~$1$. On the other, our bound guarantees that this vector has height $\leq H(\Z)$, making it better in this situation.

It also makes sense to compare our Theorem~\ref{new_siegel} to a classical result of W. M. Schmidt, specifically Theorem~2(b) of~\cite{schmidt1967}. Translating into our notation, Schmidt's theorem asserts that for any $0 \leq J < L \leq N$ our $L$-dimensional subspace $\Z \subseteq k^N$ contains a $J$-dimensional subspace $\Y \subset \Z$ with 
$$H(\Y) \leq C H(\Z)^{J/L}$$
for some constant $C$ depending on $N$ and the number field $k$. This inequality is similar to our equation~\eqref{siegel123}, where in our result Schmidt's constant $C$ disappears at the expense of a somewhat weaker upper bound for the height, as one might expect.

Our proof of this new form of Siegel's lemma does not use the Dirichlet box principle which was exploited in the earlier work of Baker \cite{baker1966}, Roth \cite{roth1955}, Siegel \cite{siegel1929}, and Thue \cite{thue1909}. Our approach also does not use methods from the geometry of numbers which were introduced in \cite{bombieri1983}; it is based solely on linear algebra. We review the necessary notation and background material in Section~\ref{notation} and present the proof of Theorem~\ref{new_siegel} in Section~\ref{newSiegelProof}. Additionally, we present a ``relative" version of our Theorem~\ref{new_siegel} in Theorem~\ref{rellem}: analogously to Theorem~12 of~\cite{bombieri1983}, we consider a finite extension $K$ of the number field $k$ and produce a small-height basis over $k$ for the null-space of a matrix defined over $K$. 

One application of our Theorem~\ref{new_siegel} together with a previous version of Siegel's lemma and certain results on integer sensing matrices (see~\cite{LDV}, \cite{konyagin}, \cite{venia_konyagin}) is the following observation.

\begin{theorem} \label{many_bases} Let $\Z \subseteq k^N$ be a subspace of dimension $L$ where $1 \leq L < N$. For each integer $M > L$ there exists a collection of vectors
$$S(M) = \left\{ \bwy_1,\dots,\bwy_M \right\} \subset \Z$$
with the following properties:

\begin{enumerate}

\item Every subcollection of $L$ vectors from $S(M)$ forms a basis for $\Z$,

\item For every $\bwy_i \in S(M)$, 
$$H(\bwy_i) \leq L^{3/2} (2M)^{\frac{L-1}{L}} \min \left\{ H(\Z)^L, \gamma_k(L)^{L/2} H(\Z) \right\},$$
where $\gamma_k(L)^{1/2}$ is the generalized Hermite's constant discussed in Section~\ref{notation}.

\end{enumerate}
\end{theorem}

\noindent
We review the necessary results on integer sensing matrices and prove Theorem~\ref{many_bases} in Section~\ref{pf_many}. In the same section, we also prove a similar result for collections of smaller linearly independent sets of bounded-height vectors in a vector space~$\Z$ (Proposition~\ref{lin_ind}). Finally, we show in Corollary~\ref{cor_sec4} that the basis guaranteed by Theorem~\ref{new_siegel} for a subspace spanned over $k$ by rows of an integer sensing matrix must satisfy ``strict monotonicity", i.e. inequalities~\eqref{siegelm113} become strict. We are now ready to proceed.

\bigskip

\section{Notation and heights}
\label{notation}

Let $k$ be a number field of degree $d := [k : \que] \geq 1$, as above, and let $r_1, r_2$ be the numbers of real and conjugate pairs of complex embeddings, respectively, so that
$$d = r_1 + 2r_2.$$
Let $\Delta_k$ be the discriminant of $k$, and write $M(k)$ for the set of places of $k$. For each $v \in M(k)$ let $d_v = [k_v : \que_v]$ be the local degree, then for each $u \in M(\que)$, $\sum_{v \mid u} d_v = d$. We select the absolute values so that $|\ |_v$ extends the usual archimedean absolute value on $\que$ when $v \mid \infty$, or the usual $p$-adic absolute value on $\que$ when $v \nmid \infty$. With this choice, the product formula reads
$$\prod_{v \in M(k)} |\beta|_v^{d_v} = 1,$$
for each nonzero $\beta \in k$. Let $N \geq 2$, and for any place $v \in M(k)$ and $\Bbeta = (\beta_1,\dots,\beta_N) \in k^N$ define the corresponding sup-norm
$$|\Bbeta|_v = \max \{ |\beta_1|_v,\dots,|\beta_N|_v \}.$$
If $v \mid \infty$, we also define Euclidean norm 
$$\|\Bbeta\|_v = \left( \sum_{i=1}^N |\beta_i|_v^2 \right)^{1/2}.$$
Then we can define the projective height function (also known as Arakelov height; see~\cite{bombieri2006}, p.66) $H : k^N \to \real_{\geq 0}$ as follows:
$$H(\Bbeta) = \left( \prod_{v \nmid \infty} |\Bbeta|_v^{d_v} \times \prod_{v \mid \infty} \|\Bbeta\|_v^{d_v} \right)^{1/d}.$$
This height is absolute, meaning that it is the same when computed over any number field $k$ containing coordinates of $\Bbeta$: this is due to the normalizing exponent~$1/d$ in the definition. 

We often use positive integers $L$ and $M$ that satisfy
\begin{equation*}\label{siegel0}
L + M = N.
\end{equation*} 
We assume that $A$ is an $M \times N$ matrix with entries in $k$ and $\rank A = M$.  Then the null-space associated to $A$ is the subspace
\begin{equation}\label{siegel5}
\Z = \big\{\Bbeta \in k^N : A \Bbeta = \bo\big\}.
\end{equation}
Obviously $\Z \subseteq k^N$ is a subspace of dimension $L$.  If the vectors in the collection
\begin{equation}\label{siegel12}
\big\{\bxi_1, \bxi_2, \dots , \bxi_L\big\}
\end{equation}
form a $k$-basis for $\Z$, we write
\begin{equation*}\label{siegel19}
X = \bigl(\bxi_1\ \bxi_2\ \cdots\ \bxi_L\bigr)
\end{equation*} 
for the corresponding $N \times L$ matrix such that the basis vectors (\ref{siegel12}) are the columns of $X$.  Then the Arakelov height of the subspace
$\Z$ is also the Schmidt height (as defined in \cite{bombieri1983} or in \cite{schmidt1967}) of the matrix $X$.  These heights are given by
\begin{equation}\label{siegel26}
H(\Z) = H(X) = H\bigl(\bxi_1 \wedge \bxi_2 \wedge \cdots \wedge \bxi_L\bigr),
\end{equation} 
where $\wedge$ stands for the usual wedge product on vectors and the height on the right of (\ref{siegel26}) is defined on the associated Grassmannian as discussed in 
\cite[section 2.8]{bombieri2006} or in \cite[section 2]{vaaler2003}: it is simply the height $H$ as above of the vector of Grassmann coordinates of $\Z$ identified with the vector space $k^{\binom{N}{L}}$. It follows using the product formula for the normalized absolute values 
$|\ |_v$, where $v$ is a place of $k$, that the value of these heights does not depend on the choice of basis.  In particular, the Arakelov
height on vectors in $k^N$ is well defined on the underlying projective space and a similar remark applies to the height on 
subspaces.  From a duality principle formulated in \cite{gordan1873} (see also~\cite{bombieri2006}, Corollary~2.8.12 on p.69 for a more modern account), we get the fundamental identity
\begin{equation}\label{siegel33}
H(\Z) = H(X) = H(A),
\end{equation}
where $H(A)$ is also the height of the $M$-dimensional subspace of $k^N$ spanned by the columns of $A^{\top}$. By a standard convention, we set $H(k^N) = H(\{\bo\}) = 1$.

If $1 \le J < L$ then a generalization of Hadamard's determinant inequality (see~\cite{schmidt1991}, Lemma~5F on p.15 and Exercise~8a on p.29) asserts that
\begin{equation}\label{siegel40}
\begin{split}
 H\bigl(\bxi_1 \wedge ~& \bxi_2 \wedge \cdots \wedge \bxi_L\bigr)\\ 
 	&\le  H\bigl(\bxi_1 \wedge \bxi_2 \wedge \cdots \wedge \bxi_J\bigr)  H\bigl(\bxi_{J + 1} \wedge \bxi_{J + 2} \wedge \cdots \wedge \bxi_L\bigr).
\end{split}
\end{equation}
Alternatively, let $Y$ and $Z$ be the $N \times J$ and $N \times (L - J)$ matrices
\begin{equation*}\label{siegel47}
Y = \bigl(\bxi_1 \ \bxi_2 \ \cdots \ \bxi_J\bigr),\quad\text{and}\quad  Z = \bigl(\bxi_{J + 1} \ \bxi_{J + 2} \ \cdots \ \bxi_L\bigr).
\end{equation*}
Then (\ref{siegel40}) is also the inequality
\begin{equation}\label{siegel54}
H(X) = H(Y\ Z)\le H(Y) H(Z)
\end{equation}
for a partitioned matrix.  By repeated application of this inequality we get
\begin{equation}\label{siegel61}
H(\Z) = H(X) \le \prod_{\ell = 1}^L H(\bxi_{\ell}).
\end{equation}
In the special case $L = N$ we have
\begin{equation*}\label{extra29}
\det\bigl(\bxi_1\ \bxi_2\ \cdots\ \bxi_N\bigr) = \bxi_1 \wedge \bxi_2 \wedge \cdots \wedge \bxi_N,
\end{equation*}
and for each $v \mid \infty$,
\begin{equation}\label{extra36}
\bigl|\det\bigl(\bxi_1\ \bxi_2\ \cdots\ \bxi_N\bigr)\bigr|_v \le \prod_{n = 1}^N \|\bxi_n \|_v,
\end{equation}
which is Hadamard's upper bound for a determinant.  As the origin of (\ref{siegel40}) and (\ref{siegel54}) is somewhat obscure,
we will follow the terminology used in \cite{fischer1908} and refer to all of these as {\it Hadamard's inequality}.  

There is another type of inequality that is satisfied by the Arakelov height on subspaces.  Let $\Z \subseteq k^N$ and $\Y \subseteq k^N$ be
$k$-linear subspaces, and let $\langle \Z, \Y \rangle$ by the subspace spanned over $k$ by $\Z \cup \Y$.  Of course $\Z \cap \Y$ is also
a subspace of $k^N$.  Then these four subspaces satisfy the inequality
\begin{equation}\label{extra43}
H\bigl(\langle \Z, \Y \rangle\bigr) H(\Z \cap \Y) \le H(\Z) H(\Y).
\end{equation}
This was proved in \cite[Theorem 1]{struppeck1990}, and it was proved independently and at about the same time by W.~M.~Schmidt.

A natural question suggested by the inequality (\ref{siegel61}) is this: if $\Z \subseteq k^N$ is a subspace of dimension $L$, does there 
exist a basis for $\Z$ such that there is nearly equality in the inequality (\ref{siegel61})?  As explained in \cite{bombieri1983}, an answer to 
this question is given by a result that is dual to Siegel's lemma.  This was proved as \cite[Theorem 8]{bombieri1983} using the Weil 
height on vectors.  An analogous result for the Arakelov height was proved as \cite[Theorem 2]{vaaler2003}.  This later result asserts that 
there exists a basis $\bbeta_1, \bbeta_2, \dots , \bbeta_L$ for $\Z$ such that
\begin{equation}\label{siegel68}
\prod_{\ell = 1}^L H\bigl(\bbeta_{\ell}\bigr) \le \gamma_k(L)^{L/2} H(\Z),
\end{equation}
where $\gamma_k(L)^{1/2}$ is a positive constant that depends on the field $k$ and the parameter $L$.  The field constant 
$\gamma_k(L)^{1/2}$ is a generalization of Hermite's constant.  It was first defined and used by J. Thunder in \cite{thunder1998}.

For each positive integer $L$ and each archimedean place $v$ of $k$ we define positive real numbers
\begin{equation*}\label{extra2}
r_v(L) = \begin{cases} \pi^{-1/2}\{\Gamma(L/2 + 1)\}^{1/L} &\text{if $v$ is a real place}\\
	(2\pi)^{-1/2}\{\Gamma(L+1)\}^{1/2L} &\text{if $v$ is a complex place,}\end{cases}
\end{equation*}
and then we define
\begin{equation*}\label{extra9}
c_k(L) = 2|\Delta_k|^{1/2d} \prod_{v|\infty} \{r_v(L)\}^{d_v/d},
\end{equation*}
where $\Delta_k$ is the discriminant of $k$.  The constant $c_k(L)$ is the normalized Haar measure of a naturally occurring subset
of the $L$-fold product of the adele ring $k_{\aA}^L$, and appears in applications of the geometry of numbers over $k_{\aA}^L$.  In 
particular, it leads to the simple upper bound for $\gamma_k(L)^{1/2}$ given by
\begin{equation}\label{extra16}
\gamma_k(L)^{1/2} \le c_k(L).
\end{equation}
A more complicated lower bound was proved in \cite[Theorem 2]{thunder1998}.  And it was shown in \cite[Theorem 3]{vaaler2003} 
that the constant $\gamma_k(L)^{L/2}$ which appears in (\ref{siegel68}) is best possible for the inequality (\ref{siegel68}).
From the lower bound proved in \cite[Theorem 2]{thunder1998} it follows that for fixed $d = [k : \Q]$ and fixed $L$, the value of 
$\gamma_k(L)^{1/2}$ grows like a small positive power of the absolute discriminant $\bigl|\Delta_k\bigr|$.  In particular, for fixed 
$d = [k : \Q]$ and fixed $L$, there are only finitely many fields $k$ such that $\gamma_k(L)^{1/2}$ is less than or equal to a positive 
real number.

If we arrange the basis vectors $\bbeta_1, \bbeta_2, \dots , \bbeta_L$ in (\ref{siegel68}) so that
\begin{equation*}\label{siegel90}
H\bigl(\bbeta_1\bigr) \le H\bigl(\bbeta_2\bigr) \le \cdots \le H\bigl(\bbeta_L\bigr),
\end{equation*}
we find that
\begin{equation*}\label{siegel97}
H\bigl(\bbeta_1\bigr) \le \gamma_k(L)^{1/2} H(\Z)^{1/L}.
\end{equation*}
As the Arakelov height of a nonzero vector is always greater than or equal to $1$, a trivial consequence of (\ref{siegel68}) is the inequality
\begin{equation}\label{siegel100}
H\bigl(\bbeta_L\bigr) \le \gamma_k(L)^{L/2} H(\Z).
\end{equation}
However, if one works over the algebraic closure $\oQ$, then it follows from the results of Roy and 
Thunder \cite{roy1996}, \cite{roy1999}, that the dependence on the number field $k$ can be removed.  Suppose, for example, that 
$\Z \subseteq \oQ^N$ is a $\oQ$-linear subspace of dimension $L$.  As a basis for $\Z$ requires only finitely many algebraic numbers 
as coordinates, there exists a number field $k$ and a collection of basis vectors for $\Z$ such that the basis vectors belong to $k^N$.
It follows that the Arakelov height of $\Z$
can be computed as before by using such a number field $k$.   Because of the way we normalize the absolute values $|\ |_v$,
where $v$ is a place of $k$, it follows in a standard manner that the Arakelov height of $\Z$ does not depend on the choice of the 
field $k$.  For such a subspace $\Z \subseteq \oQ^N$ and $\vep > 0$, it follows from \cite[Theorem 1]{roy1999} that there exists a basis 
\begin{equation}\label{siegel107}
\big\{\bzeta_1, \bzeta_2, \dots , \bzeta_L\big\} \subseteq \Z,
\end{equation}
satisfying the inequality \eqref{roy_thunder}. Again the basis vectors (\ref{siegel107}) can be arranged so that
\begin{equation*}\label{siegel111}
H\bigl(\bzeta_1\bigr) \le H\bigl(\bzeta_2\bigr) \le \cdots \le H\bigl(\bzeta_L\bigr).
\end{equation*}
Then it follows as before that
\begin{equation*}\label{siegel113}
H\bigl(\bzeta_1\bigr) \le \left( e^{\frac{(L(L-1)}{4}} + \vep \right)^{1/L} H(\Z)^{1/L},
\end{equation*}
and
\begin{equation*}\label{siegel115}
H\bigl(\bzeta_L\bigr) \le \left(e^{\frac{(L(L-1)}{4}} + \vep \right) H(\Z).
\end{equation*}
Now, that we have reviewed the properties of heights and the previously known results on small-height bases for subspaces of $k^N$, we turn to the proof of our main theorem which complements the previously known bounds.

\bigskip

\section{Proof of Theorem~\ref{new_siegel}}
\label{newSiegelProof}

Throughout this section, we let \( \bone_L\) denote an \(L\times L\) identity matrix for \(L>0\) and \(\bo_{N\times L}\) denote an \(N\times L\) matrix of all zeros for positive integers \(N,L\). Additionally, let \([N] = \{1,\ldots, N\}\) and define
  \[\J(N,L) = \left\{I \subseteq [N] : |I| = L\right\}.\]
  For an \(N\times L\) matrix \(A\) with coefficients in the number field \(k\) and for each \(I \in \J(N,L)\) define \(A_I\) to be the \(L\times L\) minor  (i.e., $L \times L$ submatrix) of \(A\) whose rows are the rows of \(A\) indexed by \(I\).

  \begin{lem}\label{new_siegel_aux}
    Let
    \begin{equation}\label{newSiegelLemMat}
      A = 
      \begin{pmatrix}
      \bone_{L-1} & \bo_{(L-1) \times 1}\\
      \bo_{1 \times (L-1)} & 1\\
      U & V
      \end{pmatrix}
    \end{equation}
    be an \(N\times L\) matrix with coefficients in \(k\) where \(U\) is an \((N-L) \times (L-1)\) matrix and \(V\) is an $(N-L) \times 1$ matrix over $k$.
   Let 
  \begin{equation}\label{newSiegel521}
    A' = 
    \begin{pmatrix}
    \bone_{L-1} \\
    \bo_{1 \times (L-1)}\\
    U
    \end{pmatrix}.
  \end{equation}
  Then 
  \begin{equation}\label{newSiegelLemEq}
    H(A) \geq H(A')
  \end{equation}
  with equality if and only if \(V = \textbf{0}_{(N-L)\times 1}\).
  \end{lem}

  \begin{proof}
  We will first partition the Grassmann coordinates of \(A'\) into two sets.
  Define
  \begin{equation*}
      R' = \left\{ I' \in \J(N,L-1): L \not\in I'\right\}
  \end{equation*}
  and
  \begin{equation*}
      T' =  \J(N,L-1) \setminus R'.
  \end{equation*}
  Define 
  \begin{equation*}
  g(A'_{R'}) = (\det A'_{I'})_{I' \in R'}
  \end{equation*} 
  to be the vector of Grassmann coordinates of \(A'\) indexed by \(R'\). Likewise, define 
  \begin{equation*}
  g(A'_{T'}) = ( \det A'_{I'})_{I' \in T'}
  \end{equation*}
  to be the vector of Grassmann coordinates of \(A'\) indexed by \(T'\).
  Notice that the \(L\)-th row of \(A'\) is identically zero, so \(\det (A'_{I'}) = 0\) when \(L \in I'\).
  Therefore, for any \(I' \in T'\), \(\det A'_{I'} = 0\) and \(g(A'_{T'}) = \bo\).
  Let \((g(A'_{R'}), g(A'_{T'}))\) denote the  vector formed by concatenating the vectors \(g(A'_{R'})\) and \(g(A'_{T'})\), so that \((g(A'_{R'}), g(A'_{T'}))\) is the vector of Grassmann coordinates of \(A'\) (up to a permutations of the Grassmann coordinates, which does not change the height). Since \(g(A'_{T'})\) is identically zero and appending additional zeros to a vector does not affect the height, we have
  \begin{equation*}
      H(A') = H(g(A'_{R'}), g(A'_{T'})) = H(g(A'_{R'})).
  \end{equation*}

  We will similarly partition the Grassmann coordinates of \(A\).
  Define
  \begin{equation*}
      R = \left\{I \in \J(N,L) : L \in I\right\},
  \end{equation*}
  \begin{equation*}
      S = \left\{\{1,2,\ldots, L-1,j\}: L+1 \leq j \leq N\right\},
  \end{equation*}
  and
  \begin{equation*}
      T = J(N,L) \setminus \left(R \cup S\right).
  \end{equation*}
Notice that for each \(I' \in R'\) there exists \(I \in R\) such that \(I = I' \cup \{L\}\).
  Therefore, there is a one-to-one correspondence between elements of \(R'\) and elements \(R\).
  In fact, there is also a one-to-one correspondence between the Grassmann coordinates of \(A'\) indexed by \(R'\) and the Grassmann coordinates of \(A\) indexed by \(R\).
  Let \(I'_1 = I' \cap \{1,\ldots, L\}\) be the set of indices of \(I'\) less than \(L\) and let \(I'_2 = I' \cap \{L+1,\ldots, N\}\) be the set indices of \(I'\) greater then \(L\), so that 
  \begin{equation*}
      A'_{I'} = \begin{pmatrix} A'_{I'_1}\\A'_{I'_2} \end{pmatrix}.
  \end{equation*}
Then \(A_I\) is of the form
  \begin{equation*}
    A_{I} = \begin{pmatrix} A'_{I'_1} & \bo_{|I'_1|\times 1}\\\bo_{1\times (L-1)}&1\\A'_{I'_2}&V \end{pmatrix},
  \end{equation*}
and so
  \begin{equation}\label{newSiegel590}
      \det A'_{I'} = \pm \det A_I.
  \end{equation}
  Since multiplying the coordinates of a vector by \(\pm 1\) does not affect the height, \eqref{newSiegel590} implies that
  \begin{equation*}
      H(g(A'_{R'})) = H(g(A_R)),
  \end{equation*}
  where analogously $g(A_{R}) = ( \det A_{I})_{I \in R}$.

  There is also a correspondence between the Grassmann coordinates of \(A\) indexed by \(S\) and the entries of \(V\).
  For each \(L+1 \leq i \leq N\) and each \(I = \{1,\ldots,L-1,i\} \in S\), \(A_I\) is of the form
  \begin{equation*}
      A_I = \begin{pmatrix}
          \bone_{L-1} & \bo_{(L-1) \times 1}\\
          \bu_{i-L} & V_{i-L}
      \end{pmatrix},
  \end{equation*}
  where $\bu_{i-L}$ is the $(i-L)$-th row of $U$ and \(V_{i-L}\) is the \((i-L)\)-th coordinate of \(V\).
  Thus \(\det A_I = V_{i-L}\) and 
  \begin{equation*}
      H(g(A_{S})) = H(V).
  \end{equation*}
  Let \((g(A_R), g(A_S), g(A_T))\) be the vector formed by concatenating \(g(A_R), g(A_S)\) and \(g(A_T)\) so that \(H(A) = H(g(A_R),g(A_S),g(A_T))\). Then at each non-archimedian place \(v \in M(k)\),
  \begin{eqnarray}\label{newSiegel613}
      |(g(A_R),g(A_S),g(A_T))|_v & = & \max \left\{|g(A_R)|_v,|g(A_S)|_v,|g(A_T)|_v \right\} \nonumber \\
      & \geq & |g(A_S)|_v = |g(A'_{S'})|_v,
  \end{eqnarray}
  and at each archimedian place \(v \in M(k)\),
  \begin{eqnarray}\label{newSiegel617}
      \|(g(A_R),g(A_S),g(A_T))\|_v^2 & = & \|g(A_R)\|_v^2+\|g(A_S)\|_v^2 + \|g(A_T)\|_v^2 \nonumber \\
      & \geq & \|g(A_S)\|_v^2 = \|g(A'_{S'})\|_v^2,
  \end{eqnarray}
  with equality in \eqref{newSiegel617} if and only if \(\|g(A_S)\|_v^2 = \|g(A_T)\|_v^2 = 0\). Combining \eqref{newSiegel613} and \eqref{newSiegel617} yields \eqref{newSiegelLemEq} where equality implies \(V\) is identically zero. As it is easy to verify that \(H(A) = H(A')\) when \(V=\bo_{N-L}\), this completes the proof.
\end{proof}

\begin{cor}\label{newSiegelCor}
    Let \(A\) be a full rank \(N\times L\) matrix with \(N>L\) and coefficients in~\(k\).
    Let \(\bomega_1,\dots,\bomega_L\) be the columns of \(A\) so that 
    \[A = \begin{pmatrix}
        \bomega_1 & \ldots & \bomega_L
    \end{pmatrix},\]
    and for some \(1 \leq i \leq L\) let 
    \[A' = \begin{pmatrix}
        \bomega_1 & \ldots & \bomega_{i-1} & \bomega_{i+1} & \ldots & \bomega_L
    \end{pmatrix}\]
    be the \(N\times (L-1)\) matrix obtained by removing the \(i\)-th column of \(A\).
    Let \(I \in \J(N,L)\) and suppose that \(A_I\) is a permutation matrix.
    Then for any \(1 \leq i \leq L\)
    \begin{equation*}
        H(A) \geq H(A')
    \end{equation*}
    with equality if and only if \(\bomega_i\) is a standard basis vector.
\end{cor}
\begin{proof}
    Let \(Q\) be an \(L\times L\) permutation matrix such that 
    \begin{equation*}
        AQ = \begin{pmatrix} \bomega_1& \ldots& \bomega_{i-1}& \bomega_{i+1}& \ldots& \bomega_{L}& \bomega_i \end{pmatrix} = \begin{pmatrix} A' & \bomega_i\end{pmatrix}.
    \end{equation*}
    Multiplying by an \(N\times N\) permutation matrix \(P_1\) on the left we can permute the rows of \(A\) so that rows of \(A\) indexed by \(I\) are swapped with the rows of \(A\) indexed by \(\{1,\ldots,L\}\).
    Then \(P_1 A Q\) is a matrix such that the first \(L\) rows form a permutation matrix.
    Multiplying again by an \(N\times N\) permutation matrix \(P_2\) on the left we can further permute the rows so that \(P_2P_1 A Q\) is a matrix so that the first \(L\) rows form an \(L\times L\) identity matrix and thus \(P_2P_1AQ\) is of the same form as the matrix in \eqref{newSiegelLemMat} and 
    \begin{equation*}
        P_2P_1 A Q = \begin{pmatrix}
            P_2P_1 A' & P_2P_1 \bomega_i
        \end{pmatrix},
    \end{equation*}
    where \(P_2P_1 A'\) is a matrix of the same form as \eqref{newSiegel521}.
    Thus we can apply Lemma \ref{new_siegel_aux} to show that
    \begin{equation}\label{newSiegel659}
        H(P_2P_1 A Q) \geq H(P_2P_1A'),
    \end{equation}
    with equality if and only if \(\bomega_i\) is a standard basis vector.
    Recall that multiplying by a permutation matrix does not change the height of a matrix. Thus, \eqref{newSiegel659} implies
    \begin{equation*}
        H(P_2P_1 A Q) = H(A) \geq H(A') = H(P_2P_1A'),
    \end{equation*}
    which completes the proof.
\end{proof}

We are now ready to prove Theorem \ref{new_siegel}.
  \begin{proof}[Proof of Theorem \ref{new_siegel}]
  Let \(A\) be an \(N\times L\) basis matrix for \(\Z\) and let \(J \in \J(N,L)\) so that \(A_J\) is a full rank minor of \(A\).
  Let \(X = A A_J^{-1}\) be another basis matrix for \(\Z\) and let \(\bomega_1,\dots,\bomega_L\) be the columns of \(X\).
  By construction, \(X_J\), the \(L\times L\) minor of \(X\) indexed by \(J\) is a permutation matrix and thus we can apply Corollary \ref{newSiegelCor} to~\(X\).
  In fact, for any \(I \subseteq \{1,\ldots,L\}\) define \(X^I\) to the matrix whose columns are the columns of \(X\) indexed by \(I\), then \(X^I\) will be a matrix such that one of the \(|I|\times |I|\) minors of \(X^I\) will be a permutation matrix.
  Therefore, for any \(I_1 \subsetneq I_2 \subset \{ 1,\dots,L \}\) we can construct a sequence of subsets
  \[I_1 = \G_1 \subsetneq \G_2 \subsetneq \ldots \subsetneq \G_{|I_2 \setminus I_1|} = I_2, \]
  where \(|\G_{i+1}\setminus \G_i| = 1\) for each \(1 \leq i < |I_2\setminus I_1|\).
  Then for each \(X^{\G_i}\) we can apply Corollary~\ref{newSiegelCor} to conclude that
   \[H(X^{I_1}) = H(X^{\G_1}) \leq \ldots \leq H(X^{\G_{|I_2\setminus I_1|}}) = H(X^{I_2}),\]
  with equality at each step if \(\bomega_{\G_{i+1}\setminus \G_i}\) is a standard basis vector.
  Since \(H(\Y_{I_1}) = H(X^{I_1})\) and \(H(\Y_{I_2}) = H(X^{I_2})\) this implies \eqref{siegelm113} which in turn implies \eqref{siegel123}.
  
Finally, from the above argument, we see that the permutation matrix $X_J = A_J A_J^{-1}$, a submatrix of the basis matrix $X = AA_J^{-1}$ indexed by $J$, is in fact the $L \times L$ identity matrix. This means that each column of $X$ has at least $L-1$ zero coordinates. Since the vectors $\bomega_1,\dots,\bomega_L$ are the columns of $X$, the basis we constructed consists of $(N - L + 1)$-sparse vectors. This completes the proof of the theorem.
\end{proof}
\smallskip

\begin{ex} \label{ex_main} We demonstrate Theorem~\ref{new_siegel} on a simple example. Let $k=\que$, $N=4$, $L=3$, and take
$$A = \begin{pmatrix} 1 & 2 & 3 \\ 4 & 3 & 1 \\ 5 & 2 & 1 \\ 2 & 1 & 3 \end{pmatrix}.$$
The corresponding vector of Grassmann coordinates is $-18 \cdot (1,1,1,1)$, and hence the height of the subspace $\Z = A\que^3 \subset \que^4$ is $H(\Z)=2$. Take the indexing set $J = \{1,2,3\} \subset \J(4,3)$ and consider the corresponding nonsingular minor
$$A_J = \begin{pmatrix} 1 & 2 & 3 \\ 4 & 3 & 1 \\ 5 & 2 & 1 \end{pmatrix},$$
then we obtain the new basis matrix
$$X = AA_J^{-1} = \begin{pmatrix} 1 & 0 & 0 \\ 0 & 1 & 0 \\ 0 & 0 & 1 \\ 1 & -1 & 1 \end{pmatrix}.$$
The column vectors $\bomega_1,\bomega_2,\bomega_3$ of $X$ are $2$-sparse and of height $\sqrt{2}$ each. Further, the $2$-dimensional subspace spanned by any two of these column vectors has height $\sqrt{3}$. For instance, take
$$\Y_{1} = \spn_{\que} \{ \bomega_1 \},\ \Y_{1,2} = \spn_{\que} \{ \bomega_1,\bomega_2 \},$$
and observe that
$$H(\Y_{\{1\}}) = \sqrt{2} < H(\Y_{\{1,2\}}) = \sqrt{3} < H(\Z) = 2.$$
\end{ex}
\smallskip

We can additionally prove a relative version of Theorem \ref{new_siegel}, analogous to Theorem 12 of \cite{bombieri1983}.
Let \(K\) be a finite extension of the number field \(k\) with \([K:k]=r\geq 2\). 
Let \(F\) be a number field such that \(k \subseteq K \subseteq F\), \(F\) is a Galois extension of \(k\) with Galois group \(G(F/k)\) and \(F\) is a Galois extension of \(K\) with Galois group \(G(F/K)\). Then \(G(F/K)\) is a subgroup of \(G(F/k)\) of index \(r\). Let \(\sigma_1,\ldots, \sigma_r \in G(F/k)\) be a set of distinct representatives of the cosets of \(G(F/K)\) in \(G(F/k)\). If \(A\) is an \(M\times N\) matrix with coefficients \(a_{mn}\) in \(K\) for \(1\leq m \leq M, 1 \leq n \leq N\), define \(\sigma_i(A) = (\sigma_i(a_{mn}))\). Then define
\begin{equation} \label{scriptA}
    \mathcal{A} = \begin{pmatrix}
    \sigma_1(A)\\\vdots\\\sigma_r(A)
    \end{pmatrix}.
\end{equation}

\begin{theorem}\label{rellem}
Let \(F,K,k\) be number fields as above. Let \(A\) be an \(M \times N\) matrix with \(rM < N\) and coefficients in \(K\) and define \(\mathcal{A}\) as above. Suppose that \(\rank \mathcal{A} = rM\) and let $L=N-rM$.  Let \(A_m\) denote the \(m\)-th row of \(A\). Then $\Z := \ker(A) \cap k^N$ is an $L$-dimensional \(k\)-vector space so that
\begin{equation}\label{rellemeq}
    H(\Z) = H(\mathcal{A}) \leq \prod_{m=1}^M H(A_m)^r.
\end{equation}
  Moreover, there exists a basis
\begin{equation*}
\big\{\bomega_1, \bomega_2, \dots , \bomega_L\big\}
\end{equation*}
for $\Z$ with the following property:  if $I \subseteq \{1, \dots , L\}$ is a nonempty subset, and 
\begin{equation*}
\Y_I = \spann_k \big\{\bomega_i : i \in I\big\}
\end{equation*}
is the $k$-linear subspace spanned by the basis vectors indexed by $I$, then
\begin{equation*}
H(\Y_I) \le \prod_{m=1}^M H(A_m)^r.
\end{equation*}
In particular, $\max_{1 \leq i \leq L} H(\bomega_i) \leq\prod_{m=1}^M H(A_m)^r$.
Moreover, if \(I_1 \subsetneq I_2 \subseteq \{1,\ldots, L\}\), then 
\begin{equation*}
  H(\Y_{I_1}) \leq H(\Y_{I_2}).
\end{equation*}
\end{theorem}

\begin{proof}
Let \(\left\{ \bzeta_1, \ldots, \bzeta_r \right\}\) be a \(k\)-basis for \(K\). Let \(A= (a_{mn})\) for \(1\leq m \leq M, 1\leq n \leq N\). We can express each coefficient $a_{mn}$ in terms of this basis as
\begin{equation*}
    a_{mn} = \sum_{j=1}^r a_{mn}^{(j)} \bzeta_j,
\end{equation*}
where \(a_{mn}^{(j)} \in k\) for all \(m,n,j\). For each \(1\leq j \leq r\), let \(A^{(j)} = (a_{mn}^{(j)})\), \(1\leq m \leq M, 1\leq n \leq N\) be an $M \times N$ matrix, and let
\begin{equation*}
    A' = \begin{pmatrix}
    A^{(1)}\\\vdots\\A^{(r)}
    \end{pmatrix}.
\end{equation*}
Then \(\bx \in k^N\) solves $A\bx = \bo$ if and only if it solves $A'\bx = \bo$, hence
\begin{equation*}
    V = \ker_{k} (A) = \ker_k(A').
\end{equation*}
Further, \(A'\) is a full-rank matrix with coefficients in \(k\) so
\begin{equation*}\label{relm1}
    H(\Z) = H(A').
\end{equation*}
Now let 
\begin{equation*}
    \Omega = (\sigma_i(\bomega_j)\bone_M)_{\substack{i = 1,\ldots, r\\ j=1, \ldots, r}}
\end{equation*}
be an \(rM\times rM\) matrix organized into \(M\times M\) blocks so that the \((i,j)\) block is \(\sigma_i(\bomega_j)\bone_M\). Then
\begin{equation*}
    \Omega A' = \mathcal{A}.
\end{equation*}
Notice that \(\Omega\) is a full-rank matrix, and so
\begin{equation}\label{relm2}
H(\mathcal{A}) = H(\Omega A') = \left( \prod_{v \in M(K)} |\det (\Omega)|_v^{[K_v:\que_v]} \right)^{\frac{1}{[K:\que]}} H(A') = H(A') = H(\Z),
\end{equation}
by the product formula. Writing $\A_i$ for the $i$-th row of $\A$, we see that \eqref{siegel61} gives
\begin{equation}\label{relm3}
    H(\mathcal{A}) \leq \prod_{i=1}^{rM} H(\mathcal{A}_i) = \prod_{i=1}^M \prod_{j=1}^r H(\sigma_j(A_i)).
\end{equation}
Since \(H(\sigma_j(A_i)) = H(A_i)\) combining (\ref{relm2}) with (\ref{relm3}) gives (\ref{rellemeq}).
We can now complete the proof by applying Theorem \ref{new_siegel} to \(\A\).
\end{proof}

\section{Proof of Theorem~\ref{many_bases}}
\label{pf_many}

To prove Theorem~\ref{many_bases}, let us first recall some results on the existence of integer sensing matrices. An $L \times M$ integer matrix $A$, $L < M$, is called an {\it integer sensing matrix for $L$-sparse signals} if every $L \times L$ submatrix of $A$ is nonsingular. This notation comes from the study of sparse signal recovery in the compressive sensing signal processing paradigm: the defining condition for $A$ ensures that for a vector $\bx \in \zed^M$ with no more than $L$ nonzero coordinates ($L$-sparse), $A\bx = \bo$ if and only if $\bx = \bo$, which allows for a unique recovery of the original sparse signal from its image under $A$. The problem of existence of such matrices $A = (a_{ij})_{1 \leq i \leq L, 1 \leq j \leq M}$ with bounded sup-norm 
$$|A| := \max \{ |a_{ij}| : 1 \leq i \leq L, 1 \leq j \leq M \}$$
and $M$ as large as possible as a function of $L$ and $|A|$ has been considered recently in~\cite{LDV} with further improvements in \cite{konyagin} and~\cite{venia_konyagin}. In particular, Theorem~2.2 of~\cite{LDV} establishes the existence of an $L \times M$ integer sensing matrix $A$ for $L$-sparse signals with $|A| = T$ and 
\begin{equation}
\label{LDV_bnd}
M \geq C L \sqrt{T}
\end{equation}
for an absolute constant $C$. Theorem~1.3 of~\cite{venia_konyagin} then establishes existence of such matrices with 
\begin{equation}
\label{VK_bnd}
M \geq \max \left\{ T+1, \frac{T^{\frac{L}{L-1}}}{2} \right\},
\end{equation}
which is an improvement on~\eqref{LDV_bnd} when $L = o(\sqrt{T})$. Combining~\eqref{LDV_bnd} with~\eqref{VK_bnd}, we obtain an $L \times M$ integer matrix $A$ with
\begin{equation}
\label{comb_bnd}
|A| = T \leq \min \left\{ \left( \frac{M}{CL} \right)^2, (2M)^{\frac{L-1}{L}}, M-1 \right\},
\end{equation}
so that any subcollection of $L$ column vectors of $A$ is linearly independent. 

\begin{proof}[Proof of Theorem \ref{many_bases}]
Now, let the notation be as in the statement of Theorem~\ref{many_bases} and let $\bomega_1, \dots, \bomega_L$ be a basis for $\Z$, a specific choice will be made later. Dividing by a nonzero coordinate, if necessary, we can assume that each $\bomega_i$ has a coordinate equal to~$1$: this operation does not change the height of these basis vectors due to the product formula. This implies that for each $1 \leq i \leq L$,
\begin{equation}
\label{ht_ineq}
\prod_{v \in M(K)} |(1,\bomega_i)|^{\frac{d_v}{d}}_v = \prod_{v \in M(K)} |\bomega_i|^{\frac{d_v}{d}}_v \leq H(\bomega_i).
\end{equation}
Write $W = (\bomega_1\ \dots\ \bomega_L)$ for the corresponding basis matrix and let $B = WA$, where $A$ is an integer sensing matrix for
$L$-sparse signals as described above. Let $S(M) = \left\{ \bwy_1,\dots,\bwy_M \right\} \subset \Z$ be the set of column vectors of $B$. Then conclusion~(1) of Theorem~\ref{many_bases} are automatically satisfied, and we only need to prove~(2). Notice that for each $1 \leq i \leq M$, 
$$\bwy_i = \sum_{j=1}^L a_{ij} \bomega_j,$$
where all $a_{ij} \in \zed$ with $|a_{ij}| \leq T$. Then Lemma~2.1 of~\cite{lenny_JNT2010} asserts that
$$H(\bwy_i) \leq L^{3/2} \max_{1 \leq j \leq L} |a_{ij}| \prod_{v \in M(K)} |(1,\bomega_i)|^{\frac{d_v}{d}}_v.$$
Combining this observation with~\eqref{ht_ineq} implies that for each $1 \leq i \leq M$,
\begin{equation}
\label{y_bnd}
H(\bwy_i) \leq L^{3/2}T \prod_{i=1}^L H(\bomega_i).
\end{equation}
Making a choice of the basis $\bomega_1, \dots, \bomega_L$ as in Theorem~\ref{new_siegel}, we obtain a bound $H(\bwy_i) \leq L^{3/2}T H(\Z)^L$. On the other hand, making a choice of the basis as in~\eqref{siegel68}, we obtain a bound $H(\bwy_i) \leq L^{3/2} \gamma_k(L)^{L/2} T H(\Z)$. The result of Theorem~\ref{many_bases} now follows by combining these observations with~\eqref{comb_bnd}, where we are choosing the bound $(2M)^{\frac{L-1}{L}}$ for $T$ since its the smallest of the three for large~$M$.
\end{proof}
\medskip

There are further results on $L \times M$ integer sensing matrix for $s$-sparse signals where $s < L$. Specifically, Theorem~1.1 of~\cite{lf_ah} asserts that for all sufficiently large $L$ there exist $L \times M$ integer sensing matrices $A$ for $s$-sparse signals with $1 \leq s \leq L-1$ such that $|A| = 2$ and
$$M \geq \left( \frac{L+2}{2} \right)^{1+\frac{2}{3s-2}}.$$
Using the same reasoning as in our argument above with $T = 2$, we immediately obtain the following observation.

\begin{prop} \label{lin_ind} For sufficiently large integers $L < N$, any integer $1 \leq s \leq L-1$, and an $L$-dimensional subspace $\Z \subseteq k^N$, there exists a collection of vectors
$$S = \left\{ \bwy_1,\dots,\bwy_M \right\} \subset \Z$$
with the following properties:

\begin{enumerate}

\item For every $\bwy_i \in S$, 
$$H(\bwy_i) \leq 2 L^{3/2} \min \left\{ H(\Z)^L, \gamma_k(L)^{L/2} H(\Z) \right\},$$
where $\gamma_k(L)^{1/2}$ is the generalized Hermite's constant discussed in Section~\ref{notation},

\item The cardinality of the set $S$ is
$$M \geq \left( \frac{L+2}{2} \right)^{1+\frac{2}{3s-2}},$$

\item Every subcollection of $s$ vectors from $S$ is linearly independent.

\end{enumerate}
\end{prop}

\noindent
Finally, we can show that the monotonicity property of the basis constructed in Theorem \ref{new_siegel} becomes strict if the subspace in question is generated by rows of a sensing matrix.

\begin{cor} \label{cor_sec4} Let \(A\) be an \(L \times M\) integer sensing matrix for $L$-sparse signals with $L < M$. Let $\Z \subset k^M$ be the $L$-dimensional subspace spanned over the field $k$ by the columns of \(A^{\top}\) and let \(\bomega_1, \ldots, \bomega_L\) be the basis for \(\Z\) constructed in Theorem~\ref{new_siegel}. Let \(\Y_I\) be as in \eqref{siegel121} for each \(I \subseteq \{1,\dots,L\}\). Then for every \(I_1 \subsetneq I_2 \subset \{1,\dots,L\}\)
$$H(\Y_{I_1}) < H(\Y_{I_2}).$$
\end{cor}

\proof
As demonstrated in the proofs of Lemmas \ref{new_siegel_aux} and \ref{newSiegelCor}, if \(H(\Y_{I_1}) = H(\Y_{I_2})\) then there must exist an \(L \times L\) minor of \(A^{\top}\) (and thus of \(A\)) of rank less than \(L\). This contradicts the fact that \(A\) is an integer sensing matrix for $L$-sparse signals.
\endproof
\bigskip

\noindent
{\bf Conflict of interest statement:} There are no conflicts of interest to be reported.
\smallskip

\noindent
{\bf Data availability statement:} Data sharing not applicable to this article as no datasets were generated or analyzed during the current study.
\smallskip

\noindent
{\bf Acknowledgement:} We would like to sincerely thank Jeff Vaaler for his many valuable comments which have very positively contributed to our paper. We are also grateful to the reviewers for their helpful comments.

\end{document}